\documentclass[11pt]{article}
\usepackage[T1]{fontenc}
\usepackage[margin=1in]{geometry}
\usepackage{amsmath,amssymb,amsthm,mathtools}
\usepackage{graphicx}
\usepackage{enumitem}
\usepackage[numbers,sort&compress]{natbib}
\usepackage[hidelinks]{hyperref}
\hypersetup{
 pdftitle={Yau's number-theoretic conjecture and sharp spectral bounds for quantum harmonic oscillators},
 pdfauthor={Zhiwen Liu, Stephen S.-T. Yau, Hao Zuo},
 pdfsubject={Sharp polynomial estimates and harmonic-oscillator applications},
 pdfkeywords={Yau conjecture, lattice points, harmonic oscillator, spectral remainder, smooth numbers}
}
\setlist[itemize]{leftmargin=1.5em,itemsep=0.25em,topsep=0.3em}
\numberwithin{equation}{section}
\allowdisplaybreaks[1]
\newtheorem{theorem}{Theorem}[section]
\newtheoremstyle{mainheading}
 {\topsep}{\topsep}{\itshape}{}{\bfseries}{.}{.5em}
 {\thmname{#1}\thmnote{ (#3)}}
\theoremstyle{mainheading}
\newtheorem{maintheorem}[theorem]{Main theorem}
\theoremstyle{plain}
\newtheorem{proposition}[theorem]{Proposition}
\newtheorem{lemma}[theorem]{Lemma}
\newtheorem{corollary}[theorem]{Corollary}
\theoremstyle{definition}

\newtheorem{example}[theorem]{Example}
\theoremstyle{remark}
\newtheorem{remark}[theorem]{Remark}
\newcommand{\MainTheoremRef}{\hyperref[thm:yau-all]{main theorem}}
\newcommand{\R}{\mathbb R}
\newcommand{\Z}{\mathbb Z}
\newcommand{\fall}[2]{\bigl(#1\bigr)_{\!\underline{#2}}}

\title{Yau's number-theoretic conjecture and sharp spectral bounds for quantum harmonic oscillators}
\author{Zhiwen Liu$^1$ \quad Stephen S.-T. Yau$^2$\thanks{Corresponding author: \href{mailto:yau@uic.edu}{yau@uic.edu}.} \quad Hao Zuo$^1$\\[0.7em]
\small $^1$Beijing Institute of Mathematical Sciences and Applications,\\
\small Beijing 100084, China\\[0.3em]
\small $^2$Department of Mathematical Sciences, Tsinghua University,\\
\small Beijing 100084, China}
\date{}

\begin{document}
\maketitle
\begin{abstract}
We prove Yau's number-theoretic conjecture on positive integral points in
real right simplices in every dimension $n\ge3$, and apply the resulting
sharp polynomial bound to quantum harmonic oscillators. The lattice-point bound holds for every choice of positive real
intercepts for which the count is nonzero, with equality precisely
when all intercepts equal the same integer. For
oscillators in dimensions $n\ge3$, we obtain explicit upper bounds for
eigenvalue counts and spectral remainders for arbitrary positive
frequencies at every threshold at or above the lowest eigenvalue.
For frequencies with rational ratios in dimensions $n\ge2$, an additional lattice coordinate and
interpolation give bounds for energy sums, or first Riesz means, at every
threshold. When all frequencies are equal (the isotropic case), the
oscillator attains these bounds: the counting
bound is exact at its energy levels, and interpolation on its natural
energy grid recovers the energy sum at every threshold. In three
dimensions, the exact formula displays the periodic contribution
associated with the resonance discussed by Fefferman and Seco. The proof
builds on the slicing method of Yau and Zhang.
Arithmetic
and geometric consequences include explicit bounds for smooth numbers
and lattice-point discrepancies.

\end{abstract}

\section{Introduction}
\label{sec:status}

Quantum harmonic oscillators connect discrete energy levels with lattice
points. Their energy levels are the eigenvalues, and a standard orthonormal
basis of energy eigenfunctions is indexed by nonnegative integer vectors.
We use this correspondence to bound quantum state counts, that is,
eigenvalue counts with multiplicity, and cumulative energy differences
relative to a threshold. For equal frequencies, the bounds are attained,
and exact formulas show how resonance produces periodic corrections
to the classical phase-space approximation.

The mathematical input is Yau's number-theoretic conjecture on positive
integer points in real right simplices. It gives a polynomial upper
bound for each simplex containing a positive integer point and specifies
exactly when equality holds. We prove it in every dimension $n\ge3$.

\paragraph*{The conjecture and the main result.}
Let
\begin{equation}
\label{eq:simplex}
 T(a_1,\ldots,a_n)
 =\left\{x\in\R_{\ge0}^n:
          \sum_{i=1}^n\frac{x_i}{a_i}\le1\right\},
 \qquad a_1\ge\cdots\ge a_n>1.
\end{equation}
The numbers $a_i$ are the intercepts on the coordinate axes.
Let $P_n(a_1,\ldots,a_n)$ count its positive integer points, including
those on the sloping boundary. Throughout, we call a simplex or slice nonempty if it contains
a positive integer point.

The starting point of our proof is the product estimate of Yau and Zhang
\cite{YauZhang2006}:
\begin{equation}
\label{eq:rough}
 n!P_n(a_1,\ldots,a_n)\le\prod_{i=1}^n(a_i-1),
 \qquad n\ge3.
\end{equation}
For equal integer intercepts $a_i=N\ge n$, elementary counting gives
$n!P_n(N,\ldots,N)=N(N-1)\cdots(N-n+1)$.
The difference from the product in \eqref{eq:rough} defines the correction
\begin{equation}
\label{eq:correction}
 \fall{v}{q}:=\prod_{j=0}^{q-1}(v-j),
 \qquad
 p_q(v):=(v-1)^q-\fall{v}{q},
 \qquad q\ge1.
\end{equation}
The conjectured polynomial bound is
\begin{equation}
\label{eq:yau-polynomial}
 g_n(a_1,\ldots,a_n)
 :=\prod_{i=1}^n(a_i-1)-p_n(a_n).
\end{equation}
For fixed $n$, let $r$ be a polynomial in one variable. If the bound
$n!P_n\le\prod_i(a_i-1)-r(a_n)$ is attained whenever all intercepts
equal an integer $N\ge n$, then $r=p_n$.

We prove Yau's conjecture \cite{YauYuanZuo2016} in every dimension $n\ge3$.
\begin{maintheorem}
\label{thm:yau-all}
Let $n\ge3$ and let $a_1\ge\cdots\ge a_n>1$ be real numbers.
If $P_n(a_1,\ldots,a_n)>0$, then
\begin{equation}
\label{eq:yau-proved}
 n!P_n(a_1,\ldots,a_n)\le g_n(a_1,\ldots,a_n).
\end{equation}
Equality holds if and only if
$a_1=\cdots=a_n=N$ for an integer $N\ge n$.
\end{maintheorem}

The condition $P_n>0$ is equivalent to
$\sum_i a_i^{-1}\le1$: every positive integer vector has each coordinate
at least one. The smallest intercept of a simplex satisfying this condition
may be arbitrarily close to one. For example,
the simplex with
\[
 a_1=\cdots=a_{n-1}=\frac{(n-1)t}{t-1},
 \qquad a_n=t,\qquad 1<t<n,
\]
contains exactly one positive integer point. The ratio of its largest
to its smallest intercept tends to infinity as $t\downarrow1$.

When all intercepts equal $n$, the simplex contains only the positive
integer point $(1,\ldots,1)$. The product estimate gives the upper bound
$(n-1)^n/n!$, whereas the \MainTheoremRef{} gives the exact value $1$.

The cases in dimensions three through eight were established in
\cite{XuYau1992,XuYau1996,LuoYauZuo2015,LinLuoYauZuo2014,
LiangYauZuo2016,YauYuanZuo2016,LinYangYauZuo2016}.
Some earlier proofs used a different polynomial, called the sharp GLY
estimate. Wang and Yau \cite{WangYau2007} found a counterexample to its
original range in dimension seven and proposed a revised range.
Our proof builds on the slicing method of Yau and Zhang
\cite{YauZhang2006} and its development by Yau, Yuan, and Zuo
\cite{YauYuanZuo2016}.
The reciprocal constraint on the last nonempty slice gives a quantitative
lower bound for the contribution of unequal intercepts. We use these terms to compensate for any negative value of the
equal-intercept part. Falling-factorial estimates then close
the induction.

\paragraph*{Spectral applications.}
Fefferman and Seco \cite[pp.~XI-4--XI-5]{FeffermanSeco1995} study
accurate sums of negative eigenvalues. They use the harmonic oscillator
to show how resonance affects lower-order spectral terms: an
oscillatory contribution can have the same order as a proposed
correction. C\'ordoba, Fefferman, and Seco
\cite[pp.~165--169]{CordobaFeffermanSeco1995} relate such questions to
lattice-point errors. We apply the \MainTheoremRef{} to the oscillator
example and its anisotropic extension, which allows unequal frequencies.

For $n\ge2$ and $0<\omega_1\le\cdots\le\omega_n$, consider the
self-adjoint harmonic oscillator on $L^2(\R^n)$,
\begin{equation}
\label{eq:oscillator}
 H_{\boldsymbol\omega}
 =\frac12\sum_{i=1}^n
   \left(-\frac{\partial^2}{\partial x_i^2}+\omega_i^2x_i^2\right),
 \qquad E_0=\frac12\sum_{i=1}^n\omega_i.
\end{equation}
Here $E_0$ is the ground-state energy, the lowest eigenvalue of
$H_{\boldsymbol\omega}$. The energy levels are
\[
 E_{\boldsymbol m}=E_0+\sum_{i=1}^n\omega_i m_i,
 \qquad \boldsymbol m\in\Z_{\ge0}^n.
\]
For each $i$, let $\phi_{m_i,\omega_i}$ be the normalized
eigenfunction of the one-dimensional oscillator
$\frac12(-d^2/dx_i^2+\omega_i^2x_i^2)$ with eigenvalue
$\omega_i(m_i+1/2)$. The functions
\[
 \psi_{\boldsymbol m}(x_1,\ldots,x_n)
 =\prod_{i=1}^n\phi_{m_i,\omega_i}(x_i),
 \qquad \boldsymbol m\in\Z_{\ge0}^n,
\]
form an orthonormal basis of $L^2(\R^n)$ consisting of
eigenfunctions of $H_{\boldsymbol\omega}$ with eigenvalues
$E_{\boldsymbol m}$.
The multiplicity of an eigenvalue $e$ is the number of indices satisfying
$E_{\boldsymbol m}=e$. Thus $\mathcal N_{\boldsymbol\omega}(E)$ counts
the indices with $E_{\boldsymbol m}\le E$, or equivalently the eigenvalues
at most $E$, with multiplicity. Shifting each $m_i$ by one gives
\begin{equation}
\label{eq:spectral-identity}
 \mathcal N_{\boldsymbol\omega}(E)
 =P_n\left(\frac{E+E_0}{\omega_1},\ldots,
           \frac{E+E_0}{\omega_n}\right),
 \qquad E\ge E_0.
\end{equation}
In the phase space $\R_x^n\times\R_\xi^n$, $x$ is position
and its Fourier-dual variable $\xi$ represents momentum.
The normalized volume below classical energy $E$ is
\[
 \mathcal W_{\boldsymbol\omega}(E)
 =\frac{\operatorname{Vol}_{2n}\!\left\{(x,\xi)\in\R^{2n}:
 \frac12\sum_i(\xi_i^2+\omega_i^2x_i^2)\le E\right\}}{(2\pi)^n}
 =\frac{E^n}{n!\prod_i\omega_i},\qquad E\ge0.
\]
Appendix~\ref{si:spectral-background} gives the
eigenfunctions, a multiplicity example, and the volume calculation.
The lattice-point description of oscillator spectra is classical; see,
for example, \cite{Larson2019}. The \MainTheoremRef{} gives the following
bound, with the Yau correction and a complete equality condition.

\begin{corollary}
\label{cor:spectral-counting}
Let $n\ge3$ and $0<\omega_1\le\cdots\le\omega_n$. For $E\ge E_0$, set
\begin{equation}
\label{eq:spectral-polynomial}
 \mathcal B_{n,\boldsymbol\omega}(E)
 =\frac1{n!}\left[
   \prod_{i=1}^n\left(\frac{E+E_0}{\omega_i}-1\right)
   -p_n\left(\frac{E+E_0}{\omega_n}\right)\right].
\end{equation}
Then
\begin{equation}
\label{eq:spectral-bound}
 \mathcal N_{\boldsymbol\omega}(E)
 \le\mathcal B_{n,\boldsymbol\omega}(E).
\end{equation}
Equality holds if and only if all frequencies equal a common
$\omega$ and $E=\omega(M+n/2)$ for an integer $M\ge0$.
Moreover,
\begin{equation}
\label{eq:spectral-remainder}
 \mathcal N_{\boldsymbol\omega}(E)-\mathcal W_{\boldsymbol\omega}(E)
 \le
 \mathcal B_{n,\boldsymbol\omega}(E)-\mathcal W_{\boldsymbol\omega}(E),
\end{equation}
and the expression on the right is a polynomial of degree $n-1$.
\end{corollary}

We also estimate the sum of the differences between the threshold $E$
and the lower energy levels. This energy sum, or first Riesz mean, is
\begin{equation}
\label{eq:riesz-definition}
 \mathcal R_{\boldsymbol\omega}(E)
 :=\sum_{\boldsymbol m}(E-E_{\boldsymbol m})_+
 =\operatorname{Tr}(E-H_{\boldsymbol\omega})_+,
 \qquad u_+:=\max(u,0).
\end{equation}
The operator $(E-H_{\boldsymbol\omega})_+$ replaces each eigenvalue
$E-E_{\boldsymbol m}$ by its positive part; its trace sums these values
with multiplicity. The identity
\[
 \mathcal R_{\boldsymbol\omega}(E)
 =E\mathcal N_{\boldsymbol\omega}(E)
  -\sum_{E_{\boldsymbol m}\le E}E_{\boldsymbol m}
\]
relates this cumulative energy difference to the state count and the
sum of energies below $E$. It is also the sum of the absolute values
of the negative eigenvalues of $H_{\boldsymbol\omega}-E$.
Frequencies are commensurate if $\omega_i=hq_i$ for some $h>0$
and positive integers $q_i$, equivalently if their ratios are rational.
The next corollary uses the grid $E_0+h\Z_{\ge0}$ to obtain a bound
at every threshold by interpolation. For equal frequencies, choosing
$h=\omega$ makes this grid the spectrum.

\begin{corollary}
\label{cor:energy-interpolation}
\label{cor:energy-sum}
Let $n\ge2$, $h>0$, and let $1\le q_1\le\cdots\le q_n$ be integers.
Set
\[
 \omega_i=hq_i,\qquad \sigma=\sum_{i=1}^nq_i,
 \qquad E_0=\frac{h\sigma}{2}.
\]
Define $V_0=0$ and, for integers $j\ge1$,
\begin{equation}\label{eq:grid-energy-polynomial}
 V_j:=\frac{h}{(n+1)!}
 g_{n+1}\left(j+\sigma,\frac{j+\sigma}{q_1},\ldots,
                         \frac{j+\sigma}{q_n}\right).
\end{equation}
For $E\ge E_0$, write
\begin{equation}\label{eq:energy-grid-position}
 \frac{E-E_0}{h}=M+\theta,\qquad
 M\in\Z_{\ge0},\quad 0\le\theta<1.
\end{equation}
Then
\begin{equation}\label{eq:interpolated-energy-bound}
 \mathcal R_{\boldsymbol\omega}(E)
 \le (1-\theta)V_M+\theta V_{M+1}.
\end{equation}
At $E=E_0$, both sides are zero. For $E>E_0$, equality holds if and
only if $q_1=\cdots=q_n=1$, for the specified choice of $h$ and $q_i$.
\end{corollary}

For $E=E_0+hj$ with integer $j\ge1$, an extra integer coordinate
expresses $\mathcal R_{\boldsymbol\omega}(E)/h$ as a lattice-point
count in dimension $n+1$. Every eigenvalue belongs to
$E_0+h\Z_{\ge0}$, so the values of $\mathcal R_{\boldsymbol\omega}$
between consecutive grid points are obtained by linear interpolation. For equal frequencies, interpolation
on the natural grid $h=\omega$ gives its exact value for every $E\ge E_0$.
Section~\ref{sec:applications} proves these statements and derives the
three-dimensional resonance formula. Figure~\ref{fig:quantum-resonance}
shows where the counting bound is exact and the periodic contribution
to the energy sum.

\paragraph*{Idea of the proof and organization.}
We prove the \MainTheoremRef{} by induction on the dimension,
starting from the three-dimensional theorem of Xu and Yau
\cite[Theorem~2.1]{XuYau1996}.
The induction uses the slicing arguments of Xu and Yau
\cite{XuYau1992,XuYau1996}, as developed further by Yau and Zhang
\cite{YauZhang2006} and Yau, Yuan, and Zuo \cite{YauYuanZuo2016}.
We group the positive integer points of the simplex according to their
last coordinate $x_n=k$. Each nonempty slice is an $(n-1)$-dimensional
simplex, and the induction hypothesis gives an upper bound for its number
of positive integer points. We show that the sum of these upper bounds
is at most $g_n(a_1,\ldots,a_n)/n!$, giving the required bound for the
total number $P_n(a_1,\ldots,a_n)$.

Sections~\ref{sec:slicing}--\ref{sec:closure} prove the main theorem,
and Section~\ref{sec:consequences} gives the arithmetic and geometric
consequences. The appendices contain the supporting estimates and
the special-family and boundary calculations.

\section{Quantum harmonic oscillators}\label{sec:applications}

\subsection{Eigenvalue counts and spectral remainders}
\label{sec:oscillator-app}

\begin{proof}[Proof of Corollary~\ref{cor:spectral-counting}]
The eigenbasis described in the introduction gives
\[
 \mathcal N_{\boldsymbol\omega}(E)
 =\#\left\{\boldsymbol m\in\Z_{\ge0}^n:
              \sum_i\omega_i m_i\le E-E_0\right\}.
\]
Set $x_i=m_i+1$. Since $\sum_i\omega_i=2E_0$, the constraint becomes
$\sum_i\omega_i x_i\le E+E_0$. This proves
\eqref{eq:spectral-identity}.
The intercepts $(E+E_0)/\omega_i$ exceed one, since
$E+E_0\ge2E_0=\sum_i\omega_i>\omega_i$.
They are in decreasing order and have reciprocal sum
$2E_0/(E+E_0)\le1$. Apply the \MainTheoremRef{} to obtain
\eqref{eq:spectral-bound}.
Equality holds exactly when these intercepts all equal $M+n$
for some integer $M\ge0$. This is equivalent to
$\omega_1=\cdots=\omega_n=\omega$ and $E=\omega(M+n/2)$.

The change of variables $y_i=\omega_i x_i$ turns the sublevel set
of $\frac12\sum_i(\xi_i^2+\omega_i^2x_i^2)$ into a $2n$-dimensional
ball, giving the volume term
$\mathcal W_{\boldsymbol\omega}(E)$ defined above.
Subtracting this term from \eqref{eq:spectral-bound} proves
\eqref{eq:spectral-remainder}. The terms of degree $n$ cancel.
The coefficient of $E^{n-1}$ in the remaining polynomial is
\[
 \frac1{n!}\left(
 \frac{(n-2)E_0}{\prod_i\omega_i}
 -\frac{n(n-3)}{2\omega_n^{n-1}}\right)
 \ge\frac1{2(n-1)!\omega_n^{n-1}}>0,
\]
since $E_0/\prod_i\omega_i\ge n/(2\omega_n^{n-1})$.
Thus the degree is exactly $n-1$.
\end{proof}

\subsection{Energy sums and an additional lattice coordinate}\label{sec:energy-app}
For arbitrary positive real frequencies and $n\ge3$, integration of
Corollary~\ref{cor:spectral-counting} gives
\begin{equation}
\label{eq:integrated-bound}
 \mathcal R_{\boldsymbol\omega}(E)
 =\int_{E_0}^E\mathcal N_{\boldsymbol\omega}(s)\,ds
 \le\int_{E_0}^E\mathcal B_{n,\boldsymbol\omega}(s)\,ds,
 \qquad E\ge E_0.
\end{equation}
The upper bound is an explicit polynomial in $E$.

We next establish the grid estimate used in
Corollary~\ref{cor:energy-interpolation}. Retain its notation
$n\ge2$, $h$, $q_i$, $\sigma$, $E_0$, and $V_j$.
Let $M\ge1$ be an integer.
Since $M$ and the $q_i$ are integers,
\begin{equation}
\label{eq:energy-lift}
 \begin{aligned}
 \frac{\mathcal R_{\boldsymbol\omega}(E_0+hM)}h
 &=\sum_{\boldsymbol m\in\Z_{\ge0}^n}
       \left(M-\sum_iq_i m_i\right)_+\\
 &=\#\left\{(\boldsymbol m,k)\in\Z_{\ge0}^{n+1}:
           \sum_iq_i m_i+k\le M-1\right\}\\
  &=P_{n+1}\left(M+\sigma,\frac{M+\sigma}{q_1},
                         \ldots,\frac{M+\sigma}{q_n}\right).
 \end{aligned}
\end{equation}
For each $\boldsymbol m$ with $\sum_iq_i m_i<M$, the extra coordinate
has exactly $M-\sum_iq_i m_i$ choices,
$k=0,\ldots,M-\sum_iq_i m_i-1$.
Shifting every coordinate by one gives the last line.
The intercepts exceed one and have reciprocal sum
$(1+\sigma)/(M+\sigma)\le1$. The \MainTheoremRef{} in dimension
$n+1$ therefore gives
\begin{equation}
\label{eq:energy-bound}
 \begin{aligned}
 \mathcal R_{\boldsymbol\omega}(E_0+hM)
 &\le V_M\\
 &=\frac{h}{(n+1)!}
 \left[(M+\sigma-1)\prod_{i=1}^n
       \left(\frac{M+\sigma}{q_i}-1\right)
       -p_{n+1}\left(\frac{M+\sigma}{q_n}\right)\right].
 \end{aligned}
\end{equation}
For every $M\ge1$, equality holds exactly when all the translated
intercepts are equal integers. Since $M+\sigma$ is an integer,
this is equivalent to $q_1=\cdots=q_n=1$.
At $M=0$, the energy sum and $V_0$ are both zero.

In the equality case,
\begin{equation}
\label{eq:energy-equality}
 \mathcal R_{(h,\ldots,h)}\bigl(h(M+n/2)\bigr)
 =h\binom{M+n}{n+1}.
\end{equation}

\subsection{Arbitrary energy thresholds for commensurate frequencies}
\label{sec:energy-interpolation}

\begin{proof}[Proof of Corollary~\ref{cor:energy-interpolation}]
For $E\ge E_0$, write $(E-E_0)/h=M+\theta$ as in
\eqref{eq:energy-grid-position}.
Since each $\sum_iq_i m_i$ is an integer, the function
$u\mapsto(u-\sum_iq_i m_i)_+$ is affine on $[M,M+1]$.
Summing over the finitely many contributing indices gives the exact identity
\begin{equation}\label{eq:energy-interpolation-identity}
 \begin{aligned}
 \mathcal R_{\boldsymbol\omega}(E)
 &=(1-\theta)\mathcal R_{\boldsymbol\omega}(E_0+hM)\\
 &\quad+\theta\mathcal R_{\boldsymbol\omega}(E_0+h(M+1)).
 \end{aligned}
\end{equation}
By \eqref{eq:energy-bound}, every positive-index grid value is at most
its corresponding $V_j$. At index zero, both are zero. This proves
\eqref{eq:interpolated-energy-bound}.
At $E=E_0$, both sides are zero. If $E>E_0$, at least one
positive-index value has positive interpolation weight. Its bound is
strict unless all $q_i=1$. When all $q_i=1$, both grid bounds are
exact, so equality holds for every $E\ge E_0$.
\end{proof}

For a fixed set of commensurate frequencies, the choice of $h$ can
be normalized by requiring $\gcd(q_1,\ldots,q_n)=1$. For equal
frequencies this gives $h=\omega$ and $q_i=1$.

\subsection{Resonance, equality, and comparison of bounds}
\label{sec:resonance}
Take equal frequencies $\omega_1=\cdots=\omega_n=\omega$ with
$n\ge2$. For $E\ge n\omega/2$, set
\[
 M=\left\lfloor\frac E\omega-\frac n2\right\rfloor,
 \qquad \theta=\frac E\omega-\frac n2-M.
\]
Then $M\in\Z_{\ge0}$, $0\le\theta<1$, and
$E=\omega(M+n/2+\theta)$.
Elementary counting and Corollary~\ref{cor:energy-interpolation} give
\begin{equation}\label{eq:equal-frequency-all-energies}
 \begin{aligned}
 \mathcal N_{(\omega,\ldots,\omega)}(E)
 &=\binom{M+n}{n},\\
 \mathcal R_{(\omega,\ldots,\omega)}(E)
 &=\omega\left[\binom{M+n}{n+1}
                    +\theta\binom{M+n}{n}\right].
 \end{aligned}
\end{equation}
Here $\binom n{n+1}=0$. The second identity follows from
\eqref{eq:energy-equality}, linear interpolation, and Pascal's identity.

Expanding the first binomial in powers of $E$ gives
\begin{equation}\label{eq:equal-count-remainder}
 \mathcal N_{(\omega,\ldots,\omega)}(E)
 -\frac{E^n}{n!\omega^n}
 =\frac{(1/2-\theta)E^{n-1}}{(n-1)!\omega^{n-1}}
   +O_{n,\omega}(E^{n-2})
 \qquad (E\to\infty).
\end{equation}
The error estimate is uniform for $0\le\theta<1$.
At the eigenvalues $E=\omega(M+n/2)$, one has $\theta=0$,
so the remainder has order $E^{n-1}$.
The equality family therefore attains the degree in
\eqref{eq:spectral-remainder}.

For $n=3$, substituting $M=E/\omega-3/2-\theta$ into the second
identity in~\eqref{eq:equal-frequency-all-energies} and expanding gives
\begin{equation}\label{eq:three-dimensional-resonance}
 \begin{aligned}
 \mathcal R_{(\omega,\omega,\omega)}(E)
 &=\frac{E^4}{24\omega^3}
   +\left(\frac{\theta(1-\theta)}4-\frac5{48}\right)
       \frac{E^2}{\omega}\\
 &\quad+\frac{\theta(1-\theta)(1-2\theta)}6 E
   +\frac{\omega}{128}(3-2\theta)(2\theta-1)^2(2\theta+1).
 \end{aligned}
\end{equation}
The coefficient of $E^2/\omega$ is $-5/48$ at $\theta=0$ and
$-1/24$ at $\theta=1/2$. It changes periodically with the position
of $E$ between energy levels. For fixed $\omega$, these two sequences
of energies show that replacing it by a constant cannot leave an
$o(E^2)$ error. Thus the first correction to the continuous
approximation $E^4/(24\omega^3)$ has a coefficient that oscillates
with the energy threshold. This is the resonant effect discussed
by Fefferman and Seco \cite[pp.~XI-4--XI-5]{FeffermanSeco1995}.

\begin{figure*}[t]
\centering
\includegraphics[width=\textwidth]{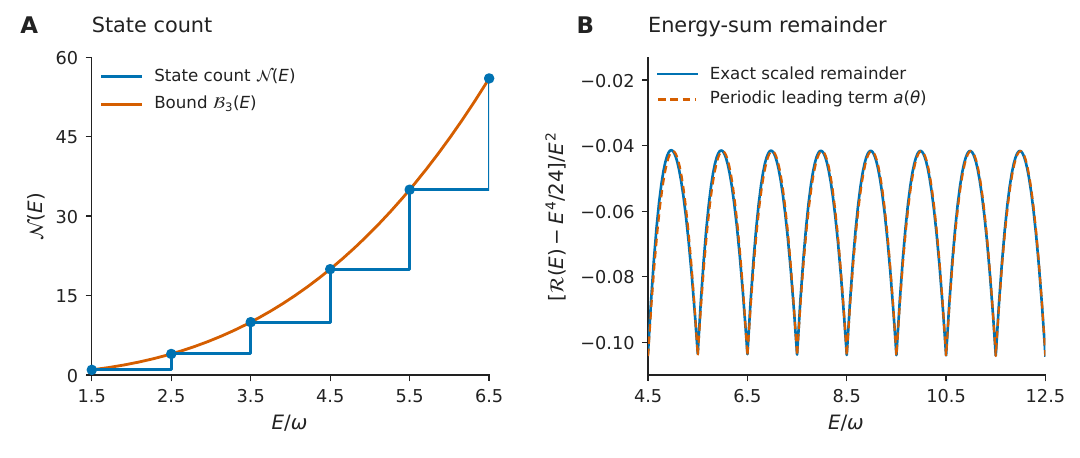}
\caption{\textbf{Finite-energy bounds and resonance in the isotropic
quantum oscillator.} Here $n=3$ and $\omega=1$.
\textbf{(A)} The exact state count is a step function. The polynomial
bound~\eqref{eq:spectral-bound} meets it at every energy level
$E=M+3/2$; the marked points indicate equality.
\textbf{(B)} The exact energy sum, after subtracting $E^4/24$ and
dividing by $E^2$, is compared with the periodic coefficient
$\theta(1-\theta)/4-5/48$, where $\theta=\{E-3/2\}$.
The difference consists of the lower-order terms in
\eqref{eq:three-dimensional-resonance} and tends to zero as
$E\to\infty$. The exact energy sum is also the interpolated upper
bound in Corollary~\ref{cor:energy-interpolation}.}
\label{fig:quantum-resonance}
\end{figure*}

A continuous integral comparison gives another upper bound, for
arbitrary positive frequencies:
\begin{equation}\label{eq:phase-energy-bound}
 \mathcal R_{\boldsymbol\omega}(E)
 \le\int_{\R_{\ge0}^n}\left(E-\sum_i\omega_i x_i\right)_+\,dx
 =\frac{E^{n+1}}{(n+1)!\prod_i\omega_i},\qquad E\ge0.
\end{equation}
The right-hand side equals
$\int_0^E\mathcal W_{\boldsymbol\omega}(s)\,ds$; we call it the
phase-space bound. To prove the inequality, note that
$x\mapsto(E-\sum_i\omega_i x_i)_+$ is convex. Jensen's
inequality on each unit cube bounds its value at the cube center by
its integral over that cube. The center values are the summands in
\eqref{eq:riesz-definition}. Summing the inequalities and evaluating
the integral proves \eqref{eq:phase-energy-bound}.

The relative sizes of the three bounds depend on the frequencies and
the energy.
For fixed unequal commensurate frequencies, the phase-space bound is
smaller than the interpolated bound at sufficiently large energies.
The same comparison holds for the integrated bound for all fixed
positive frequencies; see Appendix~\ref{si:energy-comparison}.
For example, take $\boldsymbol\omega=(1,1,2)$ and $E=6$.
Here $E_0=2$, and direct counting gives
$\mathcal R_{\boldsymbol\omega}(6)=24$. The phase-space bound is
$27$. The integrated bound is $344/9$, and the grid bound with
$h=1$ is $81/2$. For equal frequencies with $h=\omega$, the
interpolated bound gives the exact value
\eqref{eq:equal-frequency-all-energies}. Where all three estimates
apply, their minimum is also an upper bound.

\section{Slicing and the induction residual}\label{sec:slicing}

We first identify the nonempty slices, then compare the sum of their
inductive bounds with the proposed bound in dimension $n$.

\begin{lemma}\label{lem:slices}
Let $n\ge3$, $a_1\ge\cdots\ge a_n>1$, and $P_n>0$.  Put
\begin{equation}\label{eq:M}
 T=a_n\left(1-\sum_{i=1}^{n-1}\frac1{a_i}\right),
 \qquad M=\lfloor T\rfloor.
\end{equation}
Then $M\ge1$, and
\begin{equation}\label{eq:slice-sum}
 P_n(a_1,\ldots,a_n)
 =\sum_{k=1}^{M}
 P_{n-1}\!\left(
 a_1\left(1-\frac{k}{a_n}\right),\ldots,
 a_{n-1}\left(1-\frac{k}{a_n}\right)
 \right).
\end{equation}
Each slice in the sum is nonempty, and all its intercepts exceed $1$.
\end{lemma}

\begin{proof}
Slices with $x_n=k\ge a_n$ are empty, since the remaining
coordinates are positive. Fix an integer $1\le k<a_n$. The remaining
coordinates satisfy
\[
 \sum_{i=1}^{n-1}
 \frac{x_i}{a_i(1-k/a_n)}\le1.
\]
The slice contains a positive integer point exactly when it contains
$(1,\ldots,1)$. This is equivalent to
\[
 \sum_{i=1}^{n-1}\frac1{a_i(1-k/a_n)}\le1,
\]
or $k\le T$. Thus the nonempty slices are $k=1,\ldots,M$. The original
simplex contains $(1,\ldots,1)$ because $P_n>0$, so $T\ge1$.
On each nonempty $(n-1)$-dimensional slice, the reciprocal sum of the
intercepts is at most $1$. It has at least two positive terms. Every
intercept is therefore greater than $1$. The reciprocal sum
in~\eqref{eq:M} is positive, so $T<a_n$. Hence $1\le M<a_n$ and,
with the notation below, $x=a_n-M>0$.
\end{proof}

Let $d=n-1$. Write the intercepts relative to the smallest one, and
record the distance from the last nonempty slice to that intercept:
\begin{equation}\label{eq:normalized}
 t=a_n,\quad x=t-M,\quad r_i=\frac{a_i}{t},\quad
 u_k=t-k\quad(1\le k\le M).
\end{equation}
Here $r_1\ge\cdots\ge r_d\ge1$ and $u_M=x>0$.
The $k$th nonempty slice has intercepts
$(r_1u_k,\ldots,r_du_k)$.  In particular, the last slice has intercepts
$(r_1x,\ldots,r_dx)$ and satisfies
\begin{equation}\label{eq:admissibility}
 \sum_{i=1}^{d}\frac1{r_i x}\le1.
\end{equation}
Define the induction residual as the difference between the proposed
bound and the sum of the slice bounds. The factor $n$ converts the
normalization $d!$ on each slice to $n!$:
\begin{equation}\label{eq:delta}
 \Delta_M
 =g_n(r_1t,\ldots,r_dt,t)
 -n\sum_{k=1}^{M}g_d(r_1u_k,\ldots,r_du_k).
\end{equation}

\begin{proposition}\label{prop:reduction}
Assume the conjecture in dimension $d=n-1$ for all nonempty slices.
If $\Delta_M\ge0$, then \eqref{eq:yau-proved} holds for $(a_1,\ldots,a_n)$.
\end{proposition}

\begin{proof}
By Lemma~\ref{lem:slices}, the induction hypothesis applies to every
slice in~\eqref{eq:slice-sum}. Therefore
\[
 n!P_n
 =n\sum_{k=1}^{M}d!P_d(r_1u_k,\ldots,r_du_k)
 \le n\sum_{k=1}^{M}g_d(r_1u_k,\ldots,r_du_k)
 \le g_n(r_1t,\ldots,r_dt,t).
\]
\end{proof}

It remains to prove $\Delta_M\ge0$. The sum ends at the last nonempty
slice. On an empty slice the polynomial can be negative, as the
examples in Appendix~\ref{sec:audit} show.

\section{The anisotropy decomposition}\label{sec:main}

We expand the residual in the differences between the first $d$
intercepts. The root comparison below shows that every coefficient
of a nonconstant term is positive.

\subsection{A power-sum inequality and root lifting}\label{sec:root}
The next lemma extends the power-sum estimate of Yau and Zhang
\cite[Proposition~3.1]{YauZhang2006} to truncated sums of products
with nonnegative roots. Related factor-insertion comparisons appear
in \cite[Lemmas~3.3 and~3.4]{YauZhang2006}.
We give a direct proof suited to the truncation at the last nonempty slice.

\begin{lemma}\label{lem:root-lifting}
Let $d\ge2$, $t>1$, and let $M\in\Z$ satisfy $1\le M<t$.  Put $x=t-M$.
If $0\le\rho_1,\ldots,\rho_d\le x$, then
\begin{equation}\label{eq:root-lifting}
 (d+1)\sum_{k=1}^{M}\prod_{h=1}^d(t-k-\rho_h)
 <(t-1)\prod_{h=1}^d(t-\rho_h).
\end{equation}
\end{lemma}

The proof first treats zero roots by a power-sum estimate, then inserts
the roots one at a time; see Appendix~\ref{si:root}.

\subsection{Expansion in intercept differences}

Set
\begin{equation}\label{eq:anisotropy-vars}
 s=r_d,\qquad \eta_i=r_i-s\quad(1\le i\le d-1),
 \qquad \eta_d=0.
\end{equation}
The variables $\eta_i\ge0$ measure the differences from the smallest
ratio; we call them the anisotropy variables. By~\eqref{eq:admissibility},
the smallest intercept of the last slice is $sx>1$. Write
$e_j(\eta_1,\ldots,\eta_{d-1})$ for the $j$th elementary symmetric
polynomial, with $e_0=1$.

\begin{theorem}\label{thm:anisotropy}
Let $n\ge3$, $d=n-1$, and retain the notation of
Sections~\ref{sec:slicing}--\ref{sec:main}.  Define
\begin{equation}\label{eq:C}
 C_M=(t-1)(st-1)^d-p_n(t)
 -n\sum_{k=1}^{M}\fall{su_k}{d}
\end{equation}
and, for $1\le j\le d-1$,
\begin{equation}\label{eq:L}
 L_{j,M}=(t-1)t^j(st-1)^{d-j}
 -n\sum_{k=1}^{M}u_k^j(su_k-1)^{d-j}.
\end{equation}
Then
\begin{equation}\label{eq:decomposition}
 \Delta_M=C_M+
 \sum_{j=1}^{d-1}e_j(\eta_1,\ldots,\eta_{d-1})L_{j,M},
\end{equation}
and
\begin{equation}\label{eq:Lpositive}
 L_{j,M}>0\qquad(1\le j\le d-1).
\end{equation}
Consequently $\Delta_M\ge C_M$, with equality if and only if
$r_1=\cdots=r_d=s$.
\end{theorem}

\begin{proof}
We first identify the coefficients, then prove their positivity.
For a variable $v$, expand
\begin{align*}
 \prod_{i=1}^{d}(r_i v-1)
 &=\prod_{i=1}^{d}\bigl((sv-1)+\eta_i v\bigr)\\
 &=\sum_{J\subseteq\{1,\ldots,d-1\}}
 \left(\prod_{i\in J}\eta_i\right)
 v^{|J|}(sv-1)^{d-|J|}.
\end{align*}
Since the smallest intercept in the slice is $sv$,
\begin{align*}
 g_d(r_1v,\ldots,r_dv)
 &=\prod_{i=1}^{d}(r_i v-1)-(sv-1)^d+\fall{sv}{d}\\
 &=\fall{sv}{d}
 +\sum_{\varnothing\ne J\subseteq\{1,\ldots,d-1\}}
 \left(\prod_{i\in J}\eta_i\right)
 v^{|J|}(sv-1)^{d-|J|}.
\end{align*}
Moreover,
\[
 g_n(r_1t,\ldots,r_dt,t)=(t-1)\prod_{i=1}^{d}(r_i t-1)-p_n(t).
\]
Substitute these identities into~\eqref{eq:delta} and group the terms by
$j=|J|$. This gives~\eqref{eq:decomposition}.

For fixed $j$, let
\[
 T_j(v)=v^j(sv-1)^{d-j}
 =s^{d-j}v^j\left(v-\frac1s\right)^{d-j}.
\]
The roots are $0$ and $1/s<x$, with their indicated multiplicities.
Apply Lemma~\ref{lem:root-lifting} to the monic product and multiply
by the positive factor $s^{d-j}$. Since $n=d+1$, this gives
\[
 n\sum_{k=1}^{M}T_j(u_k)< (t-1)T_j(t).
\]
Thus $L_{j,M}>0$. All $\eta_i\ge0$, so every term in the anisotropy
sum is nonnegative. Since $L_{1,M}>0$, the sum vanishes only if
$e_1(\eta)=\eta_1+\cdots+\eta_{d-1}=0$. This forces all $\eta_i=0$,
which also makes every term vanish.
\end{proof}

With $d,M,x,s$ fixed, the decomposition shows that $\Delta_M$
is strictly increasing in each $\eta_i\ge0$.

\begin{remark}\label{rem:core}
The equal-intercept part $C_M$, called the core, can be negative.
Setting $\eta_1=\cdots=\eta_{d-1}=0$ may violate the nonempty-slice
condition~\eqref{eq:admissibility}. Section~\ref{sec:gain} uses this
condition to bound the anisotropy terms from below. Their sum
compensates for any negative value of the core. For example, the
admissible simplex $(a_1,a_2,a_3,a_4)=(7,7,7/2,7/2)$ has
$M=1$ and
\[
 C_M=-\frac{15}{16},\qquad \Delta_M=\frac{205}{2}.
\]
The calculation is given in Appendix~\ref{si:negative-core}.
\end{remark}

\section{A scalar lower bound for the residual}\label{sec:gain}

The last-slice constraint limits how small the anisotropy contribution
can be. We use it to replace the intercept differences by a lower bound
depending only on $d,M,x,z$.

\subsection{The reciprocal constraint and the anisotropy gain}

For the last slice, put
\begin{equation}\label{eq:A-z}
 A_i=xr_i,\qquad z=sx=A_d.
\end{equation}
Thus $A_i\ge z>1$ and $\sum A_i^{-1}\le1$.

\begin{lemma}\label{lem:fixed-minimum}
Let $A_i\ge z>1$ for $1\le i\le d$, let $A_d=z$, and assume
$\sum_iA_i^{-1}\le1$.  Define
\begin{equation}\label{eq:Phi}
 \Phi_d(z)=
 \begin{cases}
 \displaystyle
 \frac{\bigl(1+(d-2)z\bigr)^{d-1}}{(z-1)^{d-2}},
 &1<z\le d,\\[0.8em]
 (z-1)^d,&z\ge d.
 \end{cases}
\end{equation}
Then
\begin{equation}\label{eq:fixed-minimum}
 \prod_{i=1}^{d}(A_i-1)\ge\Phi_d(z).
\end{equation}
For $1<z\le d$, equality holds if and only if
\[
 A_1=\cdots=A_{d-1}=\frac{z(d-1)}{z-1},\qquad A_d=z.
\]
For $z\ge d$, equality holds if and only if all $A_i=z$.
\end{lemma}

\begin{proof}
With $A_d=z$ fixed, minimizing the product is equivalent to minimizing
the sum of the logarithms of its other factors. Put $y_i=1/A_i$ for
$i<d$ and $\varphi(y)=\log((1-y)/y)$. We have
\[
 y_i\le\min\left\{\frac1z,1-\frac1z\right\}\le\frac12,
 \qquad \sum_{i=1}^{d-1}y_i\le1-\frac1z.
\]
On $(0,1/2]$, the function $\varphi$ is decreasing and convex, since
\[
 \varphi'(y)=-\frac1{y(1-y)},\qquad
 \varphi''(y)=\frac{1-2y}{y^2(1-y)^2}\ge0.
\]
If $z\le d$, the reciprocal sum can reach $1-1/z$ under the bounds
$y_i\le1/z$. The function $\varphi$ is decreasing, so its sum is
minimized at this endpoint. By Jensen's inequality, the minimum occurs
when all $y_i$ equal $(z-1)/(z(d-1))$. If $z\ge d$, the individual
bounds already imply the sum constraint. The minimum then occurs at
$y_i=1/z$ for all $i<d$. Substitute these values and multiply by the
fixed factor $A_d-1=z-1$. This gives~\eqref{eq:Phi}. For $d>2$,
strict convexity on $(0,1/2)$ and strict decrease give the stated
equality conditions. For $d=2$, only $y_1$ varies, and strict decrease
suffices. The two configurations agree at $z=d$.
\end{proof}

\begin{proposition}\label{prop:gain}
For $1\le j\le d-1$, put
\begin{equation}\label{eq:Rratio}
 R_{j,M}=\frac{L_{j,M}}{x^j(z-1)^{d-j}}.
\end{equation}
Then
\begin{equation}\label{eq:Rordered}
 R_{1,M}>R_{2,M}>\cdots>R_{d-1,M}>0.
\end{equation}
In particular, with
\begin{equation}\label{eq:Lambda}
 \Lambda_{d,M}(x,z)
 =R_{d-1,M}=\frac{L_{d-1,M}}{x^{d-1}(z-1)},
\end{equation}
one has
\begin{align}
 \sum_{j=1}^{d-1}e_j(\eta)L_{j,M}
 &\ge \Lambda_{d,M}(x,z)
 \left\{\prod_{i=1}^{d}(A_i-1)-(z-1)^d\right\}\notag\\
 &\ge \Lambda_{d,M}(x,z)
 \left\{\Phi_d(z)-(z-1)^d\right\}.
 \label{eq:sharp-gain}
\end{align}
For $M=1$, the coefficient is
\begin{equation}\label{eq:lambda}
 \Lambda_{d,1}(x,z)
 =x\left(1+\frac1x\right)^{d-1}
 \left(1+\frac{z}{x(z-1)}\right)-(d+1).
\end{equation}
For every $M\ge1$,
\begin{equation}\label{eq:Lambda-lambda}
 \Lambda_{d,M}(x,z)\ge\Lambda_{d,1}(x,z)>0.
\end{equation}
Consequently,
\begin{equation}\label{eq:scalar-bound}
 \Delta_M\ge C_M+\Lambda_{d,M}(x,z)
 \left\{\Phi_d(z)-(z-1)^d\right\}.
\end{equation}
\end{proposition}

\begin{proof}
The ratios are positive by~\eqref{eq:Lpositive}. To order them,
compare adjacent terms for $1\le j\le d-2$. Taking a common denominator gives
\begin{align*}
 R_{j,M}-R_{j+1,M}
 =\frac{1}{x^{j+1}(z-1)^{d-j}}
 \biggl\{&(t-1)P_j(t)\\
 &-(d+1)\sum_{k=1}^{M}P_j(t-k)\biggr\},
\end{align*}
where
\[
 P_j(v)=(v-x)v^j(sv-1)^{d-j-1}.
\]
This degree-$d$ polynomial has positive leading coefficient. All its roots lie in $[0,x]$.
After division by the leading coefficient, Lemma~\ref{lem:root-lifting}
applies and makes the expression in braces strictly positive. This
proves~\eqref{eq:Rordered}.

The last-slice product expansion is
\begin{equation}\label{eq:last-product-expansion}
 \prod_{i=1}^{d}(A_i-1)-(z-1)^d
 =\sum_{j=1}^{d-1}e_j(\eta)x^j(z-1)^{d-j}.
\end{equation}
Since $R_{j,M}\ge\Lambda_{d,M}>0$ and $e_j(\eta)\ge0$, the first
inequality in~\eqref{eq:sharp-gain} follows term by term. The second
uses Lemma~\ref{lem:fixed-minimum} and $\Lambda_{d,M}>0$.

The bound $\Lambda_{d,M}\ge\Lambda_{d,1}(x,z)>0$ follows by comparing
successive values of $M$ and then setting $M=1$; the calculation is in
Appendix~\ref{si:lambda}. Theorem~\ref{thm:anisotropy}
and~\eqref{eq:sharp-gain} now give~\eqref{eq:scalar-bound}.
\end{proof}

The remaining induction step has $n\ge4$, hence $d\ge3$.
For the scalar estimates, take arbitrary parameters
$d\ge3$, $M\in\Z_{\ge1}$, $0<x\le z$, and $z>1$.
Set $s=z/x$, $t=M+x$, and $u_k=t-k$ in
\eqref{eq:C}, \eqref{eq:L}, and \eqref{eq:Lambda}.
These formulas define $C_M$ and $\Lambda_{d,M}(x,z)$ on this whole
parameter domain. Put
\begin{equation}\label{eq:E-B}
 E_d(z)=\Phi_d(z)-(z-1)^d,
 \qquad
 B_{d,M}(x,z)=C_M+\Lambda_{d,M}(x,z)E_d(z).
\end{equation}
For the parameters of a nonempty simplex,
Proposition~\ref{prop:gain} gives
\begin{equation}\label{eq:Delta-B}
 \Delta_M\ge B_{d,M}(x,z).
\end{equation}
We prove $B_{d,1}\ge0$ first, then show that $B_{d,M}$ is
nondecreasing in $M$ at fixed $d,x,z$. The following falling-factorial
estimates are used in both steps.

\subsection{Falling-factorial estimates}

\begin{lemma}
\label{lem:falling-envelopes}
Let $d\ge3$, put $A=d-1$, and define
\[
 F_d(v)=\fall{v}{d}.
\]
Recall that $p_{d+1}(v+1)=v^{d+1}-\fall{v+1}{d+1}$.
Then the following statements hold.
\begin{enumerate}[label=\textup{(\roman*)}]
\item For every integer $N\ge2$ and $1\le y\le N-1$,
\begin{equation}\label{eq:interval-product}
 |\fall{y}{N}|\le(N-1)!.
\end{equation}
\item For $v\ge A$,
\begin{equation}\label{eq:F-derivative-bound}
 F_d'(v)\le d(v-1)^{d-1}.
\end{equation}
\item If $A<z\le d$ and $0\le x\le z$, then
\begin{equation}\label{eq:pstar-envelope}
 p_{d+1}(x+1)\le p_{d+1}(z+1).
\end{equation}
\end{enumerate}
\end{lemma}

The proof uses bounds on the factors and the arithmetic--geometric mean inequality; see Appendix~\ref{si:factorials}.

\subsection{The one-slice inequality}

\begin{lemma}
\label{lem:one-slice-core}
Let $d\ge3$, $0<x\le z$, and $1<z<d$.  Put $A=d-1$ and
\[
 w=z-1+\frac zx.
\]
For $M=1$ one has
\begin{equation}\label{eq:C1-formula}
 C_1=xw^d-p_{d+1}(x+1)-(d+1)F_d(z),
\end{equation}
and
\begin{equation}\label{eq:core-lower}
 C_1\ge-(d-z)A^d.
\end{equation}
\end{lemma}

The proof separates the ranges $1<z\le d-1$ and $d-1<z<d$; see Appendix~\ref{si:one-core}.

\begin{proposition}
\label{prop:B1}
For every $d\ge3$, $0<x\le z$, and $z>1$,
\begin{equation}\label{eq:B1-positive}
 B_{d,1}(x,z)\ge0.
\end{equation}
Equality holds if and only if $x=z=d$.
\end{proposition}

For $1<z<d$, put $A=d-1$. The estimates
\[
 C_1\ge-(d-z)A^d,\qquad E_d(z)\ge (d-z)A^{d-1},\qquad
 \Lambda_{d,1}>A
\]
give $B_{d,1}>0$ for $d\ge3$. For $z\ge d$, $E_d(z)=0$
and the falling-factorial estimates give $C_1\ge0$. The full
proof and equality analysis are in Appendix~\ref{si:B1}.

\subsection{Monotonicity in the number of slices}

The next lemma is used to compare the core for successive numbers of
nonempty slices.

\begin{lemma}\label{lem:h-shape}
Let $d\ge3$, $A=d-1$, and
\begin{equation}\label{eq:h-def}
 h_d(v)=v^{d+1}-(v-1)^{d+1}-(d+1)F_d(v).
\end{equation}
Then $h_d$ is nondecreasing on $[A,\infty)$, and
\begin{equation}\label{eq:h-cap}
 h_d(v)\le h_d(A)\qquad(1<v\le A).
\end{equation}
\end{lemma}

The proof follows from Lemma~\ref{lem:falling-envelopes} and interval estimates; see Appendix~\ref{si:h-shape}.

\begin{proposition}
\label{prop:B-monotone}
Fix $d\ge3$, $0<x\le z$, and $z>1$.  For every integer $M\ge1$,
\begin{equation}\label{eq:B-monotone}
 B_{d,M+1}(x,z)\ge B_{d,M}(x,z).
\end{equation}
If $z<d$, the inequality is strict.  If $z\ge d$, equality holds in
\eqref{eq:B-monotone} exactly when $s=z/x=1$.
\end{proposition}

In these finite differences, $d,x,z$ are fixed, hence so is $s=z/x$;
$t=M+x$ increases by one with $M$.
Write $I_C=C_{M+1}-C_M$ and $I_L=L_{d-1,M+1}-L_{d-1,M}$.
Then
\[
 B_{d,M+1}-B_{d,M}
 =I_C+\frac{E_d(z)}{x^{d-1}(z-1)}I_L.
\]
Here $I_L>0$ and $E_d(z)\ge0$. If $I_C<0$, the falling-factorial
estimates give $I_C\ge-2(d+1)(d-1)!$, while the second term exceeds
$2(d+1)(d-1)!$, so the increment is positive. If $I_C\ge0$,
both contributions are nonnegative. Thus the increment is
nonnegative in either case.
The complete comparison and its equality conditions are proved
in Appendix~\ref{si:B-monotone}.

\begin{theorem}\label{thm:scalar-closure}
For every $d\ge3$, integer $M\ge1$, $0<x\le z$, and $z>1$,
\begin{equation}\label{eq:scalar-closed}
 B_{d,M}(x,z)\ge0.
\end{equation}
Equality holds if and only if $x=z=d$.
Consequently $\Delta_M\ge0$ for every admissible slice configuration.
\end{theorem}

\begin{proof}
Proposition~\ref{prop:B-monotone} and Proposition~\ref{prop:B1} give
\[
 B_{d,M}\ge B_{d,1}\ge0.
\]
If $B_{d,M}=0$, then $B_{d,1}=0$, so Proposition~\ref{prop:B1}
forces $x=z=d$. Conversely, these parameters give $B_{d,1}=0$
and $s=1$, so every subsequent increment vanishes by
Proposition~\ref{prop:B-monotone}. The last assertion follows
from~\eqref{eq:Delta-B}.
\end{proof}

\section{Proof of the main theorem}\label{sec:closure}

The induction starts from the following three-dimensional theorem.

\begin{theorem}[Xu--Yau {\cite[Theorem~2.1]{XuYau1996}}]
\label{thm:base-three}
For $a_1\ge a_2\ge a_3>1$ with $P_3>0$,
\[
 3!P_3(a_1,a_2,a_3)\le g_3(a_1,a_2,a_3),
\]
with equality if and only if $a_1=a_2=a_3$ is an integer.
\end{theorem}

Since $p_3(v)=v-1$, this bound reads
$6P_3\le(a_1-1)(a_2-1)(a_3-1)-a_3+1$.
The 1992 paper proves it for $a_3\ge2$. The cited 1996 theorem covers
$a_1\ge a_2\ge a_3\ge1$ under the same nonemptiness assumption.
It includes $1<a_3<2$, which is needed in the induction.

\begin{remark}
The corresponding two-dimensional bound is false. For example,
if $(a_1,a_2)=(100,5/2)$, then $P_2=80$ while
$g_2=150<160=2!P_2$. We therefore start with
Theorem~\ref{thm:base-three}.
\end{remark}

\begin{corollary}
\label{prop:residual-equality}
Let $d\ge3$, $M\in\Z_{\ge1}$, $t=M+x$, and $r_1\ge\cdots\ge r_d\ge1$,
with $x>0$ and $\sum_{i=1}^d(r_i x)^{-1}\le1$. Then
\begin{equation}\label{eq:residual-equality}
 \Delta_M=0
 \quad\Longleftrightarrow\quad
 r_1=\cdots=r_d=1\ \text{ and }\ x=d.
\end{equation}
\end{corollary}

\begin{proof}
Suppose $\Delta_M=0$. By~\eqref{eq:Delta-B} and
Theorem~\ref{thm:scalar-closure}, $B_{d,M}=0$ and $x=z=d$.
Hence $s=1$ and $E_d(d)=0$, so $C_M=B_{d,M}=0$.
The decomposition~\eqref{eq:decomposition} gives
\[
 0=\sum_{j=1}^{d-1}e_j(\eta)L_{j,M}.
\]
All summands are nonnegative and $L_{1,M}>0$, so
$e_1(\eta)=0$. Thus all $\eta_i=0$ and $r_1=\cdots=r_d=1$.

Conversely, if all $r_i=1$ and $x=d$, then $z=d$, all $\eta_i=0$,
and $E_d(d)=0$. By~\eqref{eq:decomposition}, \eqref{eq:E-B}, and
Theorem~\ref{thm:scalar-closure}, $\Delta_M=C_M=B_{d,M}=0$.
\end{proof}

We can now prove the \MainTheoremRef{}.

\begin{proof}[Proof of the \MainTheoremRef{}]
The case $n=3$ is Theorem~\ref{thm:base-three}. Let $n=d+1\ge4$ and
assume the theorem in dimension $d$. Lemma~\ref{lem:slices} gives
$M\ge1$ and the slice identity \eqref{eq:slice-sum}. Each nonempty
slice has intercepts greater than $1$, so the induction hypothesis
applies to every summand. The last slice satisfies
\eqref{eq:admissibility}, with $s\ge1$, $0<x\le z=sx$, and $z>1$.
Theorem~\ref{thm:scalar-closure} gives the required residual bound. Proposition~\ref{prop:reduction} then proves
\eqref{eq:yau-proved}.

To track equality, separate the total defect into the slice residual
and the defects in the lower-dimensional bounds:
\begin{align}\label{eq:total-defect}
 g_n-n!P_n
 =\Delta_M+n\sum_{k=1}^{M}
 \bigl\{g_d(r_1u_k,\ldots,r_du_k)
       -d!P_d(r_1u_k,\ldots,r_du_k)\bigr\}.
\end{align}
Every term on the right is nonnegative.  Equality in
\eqref{eq:yau-proved} therefore implies $\Delta_M=0$.  By
Corollary~\ref{prop:residual-equality}, all $r_i=1$ and $x=d$.
Hence every original intercept equals
\[
 t=M+x=M+d\in\Z.
\]
The assumption $P_n>0$ then forces $t\ge n$.

Conversely, let all intercepts equal an integer $N\ge n$. Counting
positive integer vectors with coordinate sum at most $N$ gives
\[
 P_n(N,\ldots,N)=\binom Nn,
\]
while
\[
 g_n(N,\ldots,N)=\fall{N}{n}=n!\binom Nn.
\]
This proves the equality condition and completes the induction.
\end{proof}

\section{Arithmetic and geometric consequences}\label{sec:consequences}

\subsection{Smooth numbers}\label{sec:smooth-app}
A positive integer is called $y$-smooth if all its prime factors are
at most $y$. Let $\Psi(X,y)$ count these integers up to $X$, including
$1$. Throughout this subsection, let $X\ge1$ and $y\ge5$, and let
$q_1<\cdots<q_k$ be all primes at most $y$.
Unique factorization gives an exact simplex description:
\begin{equation}
\label{eq:smooth-lattice}
 \Psi(X,y)
 =\#\left\{\boldsymbol\ell\in\Z_{\ge0}^k:
          \sum_{i=1}^k\ell_i\log q_i\le\log X\right\}.
\end{equation}
Write $L=\log X$, $\lambda_i=\log q_i$, and $S=\sum_i\lambda_i$.
The elementary volume bounds recorded by Ennola \cite{Ennola1969} are
\begin{equation}\label{eq:ennola-volume}
 \frac{L^k}{k!\prod_i\lambda_i}
 \le\Psi(X,y)\le
 \frac{(L+S)^k}{k!\prod_i\lambda_i}.
\end{equation}
The \MainTheoremRef{} gives the following bound at each finite $X$.

\begin{corollary}
\label{cor:smooth-introduction}
For $X\ge1$ and $y\ge5$, with the notation above,
\begin{equation}
\label{eq:smooth-bound-general}
 k!\Psi(X,y)
 <\prod_{i=1}^k\left(\frac{L+S}{\lambda_i}-1\right)
       -p_k\left(\frac{L+S}{\lambda_k}\right).
\end{equation}
\end{corollary}

\begin{proof}
Translation in \eqref{eq:smooth-lattice} gives
$\Psi(X,y)=P_k((L+S)/\lambda_1,\ldots,(L+S)/\lambda_k)$.
These intercepts are distinct and exceed one. Their reciprocal sum is
$S/(L+S)\le1$. Since $k\ge3$, the \MainTheoremRef{} applies. Its equality
condition gives strict inequality.
\end{proof}

Fixed-dimensional versions of this substitution appear in
\cite{LuoYauZuo2015,LinLuoYauZuo2014,LiangYauZuo2016,
YauYuanZuo2016,LinYangYauZuo2016}.

\subsection{Lattice-point discrepancies}\label{sec:geometric-app}
\begin{remark}\label{sec:remainder-app}
For $K=T(a_1,\ldots,a_n)$ with $n\ge3$, set
$\alpha=\sum_i a_i^{-1}$, $A=\prod_i a_i$, and
$\mathcal D_K(R)=\#(RK\cap\Z^n)-AR^n/n!$.
For $R\ge1$, translation gives
$\#(RK\cap\Z^n)=P_n((R+\alpha)a_1,\ldots,(R+\alpha)a_n)$.
The \MainTheoremRef{} then yields
\begin{equation}\label{eq:discrepancy-upper}
 0\le\mathcal D_K(R)\le
 \frac{g_n((R+\alpha)a_1,\ldots,(R+\alpha)a_n)-AR^n}{n!}.
\end{equation}
The upper bound has degree $n-1$, with equality exactly when all
$a_i$ equal a common $c$ and $Rc$ is an integer. For the lower bound,
the unit cubes based at the counted nonnegative integer points cover
$RK$ and have disjoint interiors. Under algebraicity and $\mathbb Q$-linear independence of the reciprocal
intercepts, Borda \cite[Theorem~8]{Borda2018} obtains a polynomial
approximation $\#(RK\cap\Z^n)=q_K(R)+\varepsilon_K(R)$ and an
asymptotic bound for $\varepsilon_K(R)$. The leading term of $q_K(R)$
is $AR^n/n!$, so the volume discrepancy is
\[
 \mathcal D_K(R)=q_K(R)-\frac{AR^n}{n!}+\varepsilon_K(R).
\]
\end{remark}

\clearpage
\appendix
\section*{Appendices}
The appendices contain supporting proofs, direct counts for special
families, and boundary examples.

\section{Power-sum and root comparison}\label{si:root}
\begin{proof}[Proof of Lemma~\ref{lem:root-lifting}]
We first prove the case $\rho_h=0$ for all $h$ by induction on the degree,
then insert the roots one at a time. Let $N=\lfloor t\rfloor$ and define
\[
 Q_d(t)=(t-1)t^d-(d+1)\sum_{k=1}^{N}(t-k)^d.
\]
Write $t=N+\beta$, $0\le\beta<1$. Direct summation gives
\[
 Q_2(t)=\beta^3-\beta^2+N\beta+\frac{N(N-1)}2>0
 \qquad(t>1).
\]
For the induction step, suppose $Q_{d-1}(t)>0$ for every $t>1$.
On each interval $(N,N+1)$,
\begin{align*}
 Q_d'(t)
 &=t^d+d(t-1)t^{d-1}
   -d(d+1)\sum_{k=1}^{N}(t-k)^{d-1}\\
 &>t^d+d(t-1)t^{d-1}-(d+1)(t-1)t^{d-1}
 =t^{d-1}>0.
\end{align*}
The function $Q_d$ is continuous at integers because each new summand
is zero there. Since $Q_d(1)=0$, we have $Q_d(t)>0$ for every $t>1$.
Since $M\le N$, discarding the remaining nonnegative summands
proves~\eqref{eq:root-lifting} when all roots are zero.

Insert the roots one at a time.  Set
\[
 P_j(v)=v^{d-j}\prod_{h=1}^{j}(v-\rho_h),
 \qquad 0\le j\le d.
\]
For $1\le k\le M$, we have $t-k\ge x\ge\rho_j$, so
\[
 0\le1-\frac{\rho_j}{t-k}\le1-\frac{\rho_j}{t}.
\]
If the result holds for $P_{j-1}$, then
\begin{align*}
 (d+1)\sum_{k=1}^{M}P_j(t-k)
 &=(d+1)\sum_{k=1}^{M}P_{j-1}(t-k)
       \left(1-\frac{\rho_j}{t-k}\right)\\
 &\le\left(1-\frac{\rho_j}{t}\right)
      (d+1)\sum_{k=1}^{M}P_{j-1}(t-k)\\
 &<\left(1-\frac{\rho_j}{t}\right)(t-1)P_{j-1}(t)
 =(t-1)P_j(t).
\end{align*}
The last inequality is strict because $\rho_j\le x<t$.
Induction on $j$ proves the claim.
\end{proof}

\section{A coefficient bound independent of the number of slices}\label{si:lambda}
\begin{proof}[Remaining estimate in the proof of Proposition~\ref{prop:gain}]
Fix $x,z$, retain $s=z/x$, and vary $M$, with $t=M+x$.
We reduce the coefficient bound to the case $M=1$ by showing that
$L_{d-1,M}$ increases with $M$. Direct subtraction gives
\[
 L_{d-1,M+1}-L_{d-1,M}
 =tT_{d-1}(t+1)-(t+d)T_{d-1}(t),
\]
where $T_{d-1}(v)=v^{d-1}(sv-1)$. Since $sx=z>1$,
\[
 \frac{T_{d-1}(t+1)}{T_{d-1}(t)}
 =\left(1+\frac1t\right)^{d-1}
  \left(1+\frac{s}{st-1}\right)
 >\left(1+\frac1t\right)^d>1+\frac dt.
\]
Thus $L_{d-1,M}$ is strictly increasing in $M$. Dividing by the
positive constant $x^{d-1}(z-1)$ gives
\[
 \Lambda_{d,M}(x,z)\ge\Lambda_{d,1}(x,z).
\]
Setting $M=1$ gives~\eqref{eq:lambda}. Moreover,
\[
 \Lambda_{d,1}(x,z)
 >x\left(1+\frac1x\right)^d-(d+1)>0.
\]
The minimum of $(x+1)^d/x^{d-1}$ occurs at $x=d-1$ and equals
$d^d/(d-1)^{d-1}>d+1$.

Therefore $\Lambda_{d,M}(x,z)\ge\Lambda_{d,1}(x,z)>0$. Finally,
Theorem~\ref{thm:anisotropy} and~\eqref{eq:sharp-gain} give
\eqref{eq:scalar-bound}.
\end{proof}

\section{Falling-factorial estimates}\label{si:factorials}
\begin{proof}[Proof of Lemma~\ref{lem:falling-envelopes}]
For~\textup{(i)}, let $m=\lfloor y\rfloor$. If $y$ is an integer,
the product is zero. Otherwise, bound the factors on each side of $y$
separately:
\[
 |\fall{y}{N}|\le(m+1)!(N-1-m)!\le(N-1)!,
\]
where the last inequality is equivalent to
$\binom{N-1}{m}\ge m+1$.

For~\textup{(ii)}, first suppose $d\ge4$. Each term of $F_d'(v)$ is a
product of $d-1$ nonnegative factors. Its arithmetic mean is largest
when the omitted factor is $v-(d-1)$. This mean is
$v-(d-2)/2\le v-1$. The arithmetic--geometric mean inequality bounds
each term by $(v-1)^{d-1}$. There are $d$ terms. For $d=3$, use
\[
 F_3'(v)=3v^2-6v+2\le3(v-1)^2.
\]

For~\textup{(iii)}, we bound the first interval directly and prove
monotonicity on the remaining two intervals. First take $0\le x\le A-1$.
Part~\textup{(i)} and $d!\le A^d$ give
\[
 p_{d+1}(x+1)\le(A-1)^{d+1}+d!\le A^{d+1}=p_{d+1}(A+1).
\]
Indeed, $A!\le A^{d-2}$ and $d\le A^2$ give $d!\le A^d$,
while $(A-1)^{d+1}\le(A-1)A^d$.
On $A-1\le x\le A$, write
\[
 \fall{x+1}{d+1}=Q(x)(x-A),
 \qquad Q(x)=\prod_{j=-1}^{A-1}(x-j).
\]
Here $Q,Q'\ge0$. The $d$ factors of $Q$ have arithmetic mean
$x-(d-3)/2\le x$, so $Q(x)\le x^d$. Since
\[
 \frac{d}{dx}\fall{x+1}{d+1}=Q(x)+(x-A)Q'(x)\le Q(x),
\]
the derivative of $x\mapsto p_{d+1}(x+1)$ satisfies
\[
 p_{d+1}'(x+1)=(d+1)x^d-\frac{d}{dx}\fall{x+1}{d+1}
 \ge d x^d\ge0.
\]
Finally, let $A\le x\le d$. All factors are nonnegative. The derivative
of $\fall{x+1}{d+1}$ has $d+1$ terms, each a product of $d$ factors.
Their arithmetic mean is at most $x-(d-3)/2\le x$. The derivative
is therefore at most $(d+1)x^d$, so the function
$x\mapsto p_{d+1}(x+1)$ is also nondecreasing on this interval.
The three intervals prove~\eqref{eq:pstar-envelope}.
\end{proof}

\section{The one-slice core estimate}\label{si:one-core}
\begin{proof}[Proof of Lemma~\ref{lem:one-slice-core}]
Substituting $t=x+1$ and $s=z/x$ into~\eqref{eq:C}
gives~\eqref{eq:C1-formula}. First suppose $1<z\le A$.
Since
\[
 w-x=\frac{(z-x)(x+1)}x\ge0,
\]
the term $x(w^d-x^d)$ in the alternative expression
\[
 C_1=x(w^d-x^d)+\fall{x+1}{d+1}-(d+1)F_d(z)
\]
is nonnegative.  Lemma~\ref{lem:falling-envelopes}\textup{(i)} gives
\[
 |\fall{x+1}{d+1}|\le d!,
 \qquad |F_d(z)|\le A!.
\]
For $d\ge4$, since $d-z\ge1$,
\begin{align*}
 C_1+(d-z)A^d
 &\ge A^d-d!-(d+1)A!\\
 &=A^d-(2d+1)A!\ge0.
\end{align*}
The last inequality follows from $A!\le A^{d-2}$ and
$A^2\ge2d+1$.  When $d=3$, $F_3(z)\le0$ and
\[
 \fall{x+1}{4}+1=(x^2-x-1)^2\ge0,
\]
so $C_1+8(3-z)\ge-1+8>0$.

It remains to consider $A<z<d$.  Lemma~\ref{lem:falling-envelopes}
gives $p_{d+1}(x+1)\le p_{d+1}(z+1)$. For fixed $z$,
\[
 \frac{d}{dx}\left[x\left(z-1+\frac zx\right)^d\right]
 =w^{d-1}\left(z-1-\frac{(d-1)z}{x}\right)<0
 \qquad(0<x\le z<d).
\]
Consequently $x(z-1+z/x)^d\ge z^{d+1}$, and
\begin{align*}
 C_1
 &\ge z^{d+1}-p_{d+1}(z+1)-(d+1)F_d(z)\\
 &=(z-d)F_d(z).
\end{align*}
Here $F_d(z)\ge0$. To prove $F_d(z)\le A^d$, pair the factors $z$ and
$z-A$. Their product is at most $d\le A^2$. Each of the remaining
$d-2$ factors is at most $A$. Since $z-d<0$, this upper bound for
$F_d(z)$ gives the lower bound~\eqref{eq:core-lower}.
\end{proof}

\section{The one-slice lower bound}\label{si:B1}
\begin{proof}[Proof of Proposition~\ref{prop:B1}]
Suppose first that $1<z<d$, and put $A=d-1$.
We bound $E_d(z)$ and $\Lambda_{d,1}(x,z)$ from below to offset the
possible negative core contribution. Logarithmic differentiation of
the first branch in~\eqref{eq:Phi}
gives
\[
 \frac{\Phi_d'(z)}{\Phi_d(z)}
 =\frac{(d-2)(z-d)}{(1+(d-2)z)(z-1)}\le0
 \qquad(1<z<d).
\]
Since $\Phi_d(d)=A^d$ and $d-z=A-(z-1)$, the difference-of-powers
identity gives
\begin{equation}\label{eq:E-lower}
 E_d(z)\ge A^d-(z-1)^d\ge (d-z)A^{d-1}.
\end{equation}
Moreover, the $M=1$ formula in Proposition~\ref{prop:gain} gives
\begin{align}
 \Lambda_{d,1}(x,z)
 &>x\left(1+\frac1x\right)^d-(d+1)\notag\\
 &\ge\frac{d^d}{A^A}-(d+1)\ge A.
 \label{eq:Lambda-A}
\end{align}
Indeed, $(x+1)^d/x^{d-1}$ is minimized at $x=A$, and
$(1+1/A)^A\ge2$.

Lemma~\ref{lem:one-slice-core} and
\eqref{eq:E-lower}--\eqref{eq:Lambda-A} yield
\[
 B_{d,1}>-(d-z)A^d+A(d-z)A^{d-1}=0.
\]
Thus~\eqref{eq:B1-positive} is strict whenever $1<z<d$.

Now assume $z\ge d$. Then $E_d(z)=0$, and we need $C_1\ge0$.
Fix $x$ and write $z=sx$. Regard
\[
 C_1(x,s)=x[s(x+1)-1]^d-p_{d+1}(x+1)-(d+1)F_d(sx)
\]
as a function of $s$.  Whenever $sx\ge d$, Lemma
\ref{lem:falling-envelopes}\textup{(ii)} gives
\begin{align*}
 \frac{\partial C_1}{\partial s}
 &>dx(sx-1)^{d-1}
 \left[(x+1)\left(1+\frac1x\right)^{d-1}-(d+1)\right]>0.
\end{align*}
The last bracket is positive because $(x+1)^d/x^{d-1}$ has minimum
$d^d/A^A>d+1$.

If $x\ge d$, then $s\ge1$, and the comparison point $s=1$ lies
in the region $sx\ge d$. Monotonicity in $s$ gives
\[
 C_1(x,s)\ge C_1(x,1)
 =\fall{x}{d+1}\ge0.
\]
The first inequality is strict for $s>1$. The falling factorial
vanishes at $x=d$ and is positive for $x>d$. Thus equality in this
case holds exactly when $x=d$ and $s=1$.

If $x<d$, compare instead with $s=d/x$, the point where $sx=d$.
By Lemma~\ref{lem:falling-envelopes}\textup{(iii)},
\[
 p_{d+1}(x+1)\le p_{d+1}(d+1)=d^{d+1}-(d+1)!.
\]
For $0<x<d$,
\[
 \frac{d}{dx}\left[x\left(A+\frac dx\right)^d\right]
 =(d-1)\left(A+\frac dx\right)^{d-1}
       \left(1-\frac dx\right)<0.
\]
Comparison with $x=d$ gives $x(A+d/x)^d>d^{d+1}$.
It follows that
\[
 C_1(x,d/x)>d^{d+1}
 -\{d^{d+1}-(d+1)!\}-(d+1)d!=0.
\]
Together with the case $x\ge d$, this shows that equality holds
exactly when $x=z=d$.
\end{proof}

\section{The auxiliary finite-difference function}\label{si:h-shape}
\begin{proof}[Proof of Lemma~\ref{lem:h-shape}]
By Lemma~\ref{lem:falling-envelopes}\textup{(ii)}, for $v\ge A$,
\[
 h_d'(v)=(d+1)\{v^d-(v-1)^d-F_d'(v)\}\ge0,
\]
because $v^d-(v-1)^d\ge d(v-1)^{d-1}$.

For~\eqref{eq:h-cap}, write
$P(v)=v^{d+1}-(v-1)^{d+1}$, so $h_d=P-(d+1)F_d$ and
$h_d(A)=P(A)$.  If $F_d(v)\ge0$, monotonicity of $P$ gives the result.
Suppose $F_d(v)<0$. It suffices to prove $P(A)-P(v)\ge(d+1)|F_d(v)|$. For $d\ge4$ one has
\begin{equation}\label{eq:factorial-step}
 (A-1)^d-(A-2)^d\ge A!.
\end{equation}
For $d=4$, the inequality follows by substitution. For $d\ge5$, the
left side is at least $d(d-3)^{d-1}$. Now use
$A!\le A(A-1)(d-3)^{d-3}$ and $d(d-3)^2\ge A(A-1)$.
If $v\le A-1$, integration of $P'$ over $[A-1,A]$ and
Lemma~\ref{lem:falling-envelopes}\textup{(i)} give
\begin{align*}
 P(A)-P(v)
 &\ge P(A)-P(A-1)\\
 &\ge(d+1)\{(A-1)^d-(A-2)^d\}\\
 &\ge(d+1)A!\ge(d+1)|F_d(v)|.
 \end{align*}
If $A-1\le v\le A$, factor out $v-A$ from $F_d(v)$. The absolute
values of the remaining factors have product at most $A!$, so
$|F_d(v)|\le(A-v)A!$. The same integral estimate on $[v,A]$ gives
$P(A)-P(v)\ge(d+1)(A-v)A!$.
This proves~\eqref{eq:h-cap} for $d\ge4$. For $d=3$, use
$h_3(v)=6v^2-4v-1$, which is increasing for $v>1$.
\end{proof}

\section{The slice increment}\label{si:B-monotone}
\begin{proof}[Proof of Proposition~\ref{prop:B-monotone}]
Fix $d,x,z$, so $s=z/x$ is fixed. In passing from $M$ to $M+1$,
$t=x+M$ increases by one.
We compute the increments of the core and its coefficient separately,
then use the coefficient term to control any decrease in the core.
Use $T_{d-1}(v)=v^{d-1}(sv-1)$ from the coefficient comparison.
Direct subtraction in~\eqref{eq:C} and~\eqref{eq:L} gives
\begin{align}
 I_C:=C_{M+1}-C_M
={}&t[s(t+1)-1]^d-(t-1)(st-1)^d\notag\\
 &-\{t^{d+1}-(t-1)^{d+1}\}
 +(d+1)F_d(t)-(d+1)F_d(st),
 \label{eq:IC}\\
 I_L:=L_{d-1,M+1}-L_{d-1,M}
={}&tT_{d-1}(t+1)-(t+d)T_{d-1}(t).
 \label{eq:IL}
\end{align}
Thus
\begin{equation}\label{eq:B-increment}
 B_{d,M+1}-B_{d,M}
 =I_C+\frac{E_d(z)}{x^{d-1}(z-1)}I_L.
\end{equation}

Binomial expansion and~\eqref{eq:h-def} give
\begin{equation}\label{eq:IC-envelope}
\begin{aligned}
 I_C={}&h_d(st)-h_d(t)\\
 &+st\sum_{j=2}^{d}\binom dj(st-1)^{d-j}(s^{j-1}-1)
 \ge h_d(st)-h_d(t).
\end{aligned}
\end{equation}
Here $t>1$ and $s\ge1$; the inequality is strict if $s>1$.

First assume $1<z<d$.
Lemma~\ref{lem:h-shape} now shows that $I_C\ge0$ unless
\begin{equation}\label{eq:bad-strip}
 1<t\le st<A=d-1.
\end{equation}
If $t\ge A$, both arguments lie in the interval where $h_d$ is
nondecreasing. If $t<A\le st$, then
$h_d(st)\ge h_d(A)\ge h_d(t)$. It remains to consider
\eqref{eq:bad-strip}. Since $v^{d+1}-(v-1)^{d+1}$ is increasing for $v>1$,
\begin{align*}
 I_C&\ge h_d(st)-h_d(t)\\
 &\ge-(d+1)\{|F_d(st)|+|F_d(t)|\}.
\end{align*}
Lemma~\ref{lem:falling-envelopes}\textup{(i)} therefore gives the uniform
estimate
\begin{equation}\label{eq:IC-negative}
 I_C\ge-2(d+1)A!.
\end{equation}

On the other hand,
\[
 \frac{T_{d-1}(t+1)}{T_{d-1}(t)}
 =\left(1+\frac1t\right)^{d-1}
  \left(1+\frac{s}{st-1}\right)
 >\left(1+\frac1t\right)^d,
\]
so
\begin{equation}\label{eq:IL-lower}
 I_L>\binom d2 t^{d-2}(st-1)>0.
\end{equation}
If $I_C<0$, then~\eqref{eq:bad-strip} holds. Since
$st=z+sM\ge z+M>2$, we have $d\ge4$ and $z+M<A$.
Thus $A-M>z>0$, and the following denominators are positive.
The middle comparison uses the decrease of $(d-z)/z$ in $z$;
the last uses $M\ge1$:
\begin{equation}\label{eq:czx}
 \frac{d-z}{x}\ge\frac{d-z}{z}
 >\frac{M+1}{d-1-M}\ge\frac2{d-2}.
\end{equation}
Equations~\eqref{eq:E-lower}, \eqref{eq:IL-lower}, and
$t>x$, $st>z$ imply
\begin{align*}
 \frac{E_d(z)}{x^{d-1}(z-1)}I_L
 &>\binom d2 A^{d-1}\frac{d-z}{x}\\
 &>\frac{dA^d}{d-2}
 \ge2(d+1)A!.
\end{align*}
For the last inequality, use $A!\le A^{d-2}$ and
\[
 dA^2-2(d+1)(d-2)=d(d-2)^2-d+4>0\qquad(d\ge4).
\]
Combining this with~\eqref{eq:IC-negative} gives
a positive increment. If $I_C\ge0$, the increment is positive because
$E_d(z)>0$ and $I_L>0$. This proves strict inequality for $z<d$.

Finally, if $z\ge d$, then $E_d(z)=0$ and $st=z+sM\ge d+1>A$.
According as $t\ge A$ or $t<A$, apply the monotonicity or the cap in
Lemma~\ref{lem:h-shape}; together with~\eqref{eq:IC-envelope},
these give $I_C\ge0$.
The inequality is strict when $s>1$.
When $s=1$, the same identity gives $I_C=0$.
\end{proof}

\section{A negative core on an admissible simplex}\label{si:negative-core}
\begin{example}\label{ex:negative-core}
Let $n=4$, $d=3$, and take
\[
 (a_1,a_2,a_3,a_4)=\left(7,7,\frac72,\frac72\right).
\]
Then $T=3/2$, so $M=1$, and the normalized parameters are
\[
 t=\frac72,\qquad x=\frac52,\qquad s=1,\qquad
 (r_1,r_2,r_3)=(2,2,1).
\]
The last-slice reciprocal sum is $2/x=4/5<1$. Direct substitution in
\eqref{eq:C} and~\eqref{eq:delta} gives
\[
 C_M=-\frac{15}{16}<0,\qquad \Delta_M=\frac{205}{2}>0.
\]
Here the symmetric part is negative, while the anisotropy contribution
makes the residual positive.
\end{example}

\section{Hermite eigenfunctions and sublevel volumes}
\label{si:spectral-background}

For the operator~\eqref{eq:oscillator}, define the Hermite
polynomials by
\[
 H_m(y)=(-1)^m e^{y^2}\frac{d^m}{dy^m}(e^{-y^2}),
 \qquad m\in\Z_{\ge0}.
\]
They satisfy the orthogonality relation
\[
 \int_{\R}e^{-y^2}H_m(y)H_\ell(y)\,dy
 =2^m m!\sqrt{\pi}\,\delta_{m\ell}.
\]
Thus the normalized one-dimensional eigenfunctions are
\[
 \varphi_{m,\omega}(x)
 =\frac{(\omega/\pi)^{1/4}}{\sqrt{2^m m!}}\,
   H_m(\sqrt{\omega}\,x)e^{-\omega x^2/2},
 \qquad m\in\Z_{\ge0}.
\]
The Hermite differential equation gives
\[
 \frac12\left(-\frac{d^2}{dx^2}+\omega^2x^2\right)
 \varphi_{m,\omega}
 =\omega\left(m+\frac12\right)\varphi_{m,\omega}.
\]
The products
\[
 \psi_{\boldsymbol m}(x)
 =\prod_{i=1}^n\varphi_{m_i,\omega_i}(x_i)
\]
form an orthonormal eigenbasis of $H_{\boldsymbol\omega}$, with eigenvalues
$E_{\boldsymbol m}=\sum_i\omega_i(m_i+1/2)$.
The lowest eigenvalue $E_0$ occurs only at $\boldsymbol m=0$ and is simple.
The multiplicity of any eigenvalue is the number of multi-indices giving
that value.

For example, when $n=2$ and $\omega_1=\omega_2=\omega$,
the eigenvalue $\omega(k+1)$ has multiplicity $k+1$ for $k\in\Z_{\ge0}$.
At $E=3\omega$, the eigenvalues $\omega$, $2\omega$, and $3\omega$
therefore give
\[
 \mathcal N_{(\omega,\omega)}(3\omega)=1+2+3=6,
 \qquad
 \mathcal R_{(\omega,\omega)}(3\omega)
 =2\omega+2\omega=4\omega.
\]
The eigenvalue $3\omega$ contributes $3$ to $\mathcal N$ and zero
to $\mathcal R$.

For the quadratic form
\[
 h_{\boldsymbol\omega}(x,\xi)
 =\frac12\sum_i(\xi_i^2+\omega_i^2x_i^2),
 \qquad (x,\xi)\in\R^n\times\R^n,
\]
set $y_i=\omega_i x_i$.
Then $dx\,d\xi=(\prod_i\omega_i)^{-1}dy\,d\xi$, and
$\{h_{\boldsymbol\omega}\le E\}$ becomes a ball of radius $\sqrt{2E}$
in $\R^{2n}$. Hence
\[
 \frac{\operatorname{Vol}_{2n}\{h_{\boldsymbol\omega}\le E\}}
      {(2\pi)^n}
 =\frac{\pi^n(2E)^n}
        {(2\pi)^n n!\prod_i\omega_i}
 =\frac{E^n}{n!\prod_i\omega_i},
 \qquad E\ge0.
\]
The right-hand side is $\mathcal W_{\boldsymbol\omega}(E)$ defined
in the main text.

\section{Comparison of the energy bounds at high energies}
\label{si:energy-comparison}
Retain the notation of Corollary~\ref{cor:energy-interpolation}, with
$n\ge2$, and fix $h$ and the integers $q_i$. Put
$U(E)=(1-\theta)V_M+\theta V_{M+1}$, where
$E=E_0+h(M+\theta)$. Expanding~\eqref{eq:grid-energy-polynomial} gives
\[
 U(E)=\frac{E^{n+1}}{(n+1)!\prod_i\omega_i}
       +c_{\boldsymbol q}E^n+O(E^{n-1}),
\]
where
\[
 c_{\boldsymbol q}
 =\frac{h^{1-n}}{2(n+1)!}
 \left[\frac{(n-1)\sigma-2}{\prod_iq_i}
       -\frac{(n+1)(n-2)}{q_n^n}\right].
\]
Linear interpolation over intervals of fixed length $h$ changes only
the $O(E^{n-1})$ term, since the second derivative of the polynomial
has degree $n-1$. The remainder is uniform for $0\le\theta<1$.
If $\boldsymbol q\ne(1,\ldots,1)$, then $\sigma>n$ and
$\prod_iq_i\le q_n^n$, so
\[
 c_{\boldsymbol q}\ge
 \frac{h^{1-n}(n-1)(\sigma-n)}{2(n+1)!q_n^n}>0.
\]
For arbitrary positive frequencies and $n\ge3$, let $c>0$ be the
coefficient of $E^{n-1}$ in
$\mathcal B_{n,\boldsymbol\omega}(E)-\mathcal W_{\boldsymbol\omega}(E)$,
computed in the main text, Section~\ref{sec:oscillator-app}. Then
\[
 \int_{E_0}^E\mathcal B_{n,\boldsymbol\omega}(s)\,ds
 =\frac{E^{n+1}}{(n+1)!\prod_i\omega_i}
   +\frac cn E^n+O(E^{n-1}).
\]
Thus the phase-space bound~\eqref{eq:phase-energy-bound} is strictly
smaller than the integrated bound for all sufficiently large $E$.
It is also strictly smaller than the interpolated bound at sufficiently
high energies whenever $\boldsymbol q\ne(1,\ldots,1)$.

\section{Explicit subfamilies}\label{sec:subfamilies}

We give direct counting proofs for three families of simplices.
They also show the equality condition in the \MainTheoremRef{}.

\begin{proposition}\label{thm:equal}
Let $n\ge2$ and $a\ge n$.  Then
\[
 P_n(a,\ldots,a)=\binom{\lfloor a\rfloor}{n}
\]
and
\[
 n!P_n(a,\ldots,a)\le\fall{a}{n}=g_n(a,\ldots,a).
\]
Equality holds if and only if $a$ is an integer.
\end{proposition}

\begin{proof}
Set $m=\lfloor a\rfloor$. The lattice condition becomes
$x_1+\cdots+x_n\le m$. Counting the positive integer solutions for each
possible sum gives
\[
 P_n=\sum_{q=n}^{m}\binom{q-1}{n-1}=\binom{m}{n}.
\]
Thus $n!P_n=\fall{m}{n}\le\fall{a}{n}$, with equality precisely when
$a=m$.
\end{proof}

\begin{proposition}\label{thm:first-chamber}
Let
\[
 a_1\ge\cdots\ge a_n=t>1,
 \qquad m=\lfloor t\rfloor\ge n,
\]
and assume
\begin{equation}\label{eq:first-chamber}
 a_1<\frac{mt}{t-1}.
\end{equation}
Then $P_n(a_1,\ldots,a_n)=\binom{m}{n}$, and the Yau bound holds.
Equality holds if and only if $a_1=\cdots=a_n=m\in\Z$.
\end{proposition}

\begin{proof}
Let $S=x_1+\cdots+x_n$.  If $S\le m$, then
\[
 \sum_i\frac{x_i}{a_i}\le\frac{S}{t}\le1.
\]
Conversely, if $S\ge m+1$, then
\[
 \sum_i\frac{x_i}{a_i}
 \ge\frac{S-x_n}{a_1}+\frac{x_n}{t}
 \ge\frac{m}{a_1}+\frac1t>1,
\]
where the second inequality uses $S\ge m+1$, $x_n\ge1$, and
$a_1\ge t$, while the last is equivalent to~\eqref{eq:first-chamber}.
Thus the counted points are exactly the positive integer vectors whose
coordinates sum to at most $m$. Also,
\[
 g_n(a_1,\ldots,a_n)
 \ge\fall{t}{n}\ge\fall{m}{n}=n!\binom{m}{n}.
\]
Equality in both inequalities holds exactly when
$a_1=\cdots=a_n=m$.
\end{proof}

\begin{proposition}\label{thm:one-axis}
Let $n\ge3$, let $q\ge n$ be an integer, and let $b\ge q$ be real.  Then
the Yau bound holds for
\[
 (a_1,a_2,\ldots,a_n)=(b,q,\ldots,q).
\]
Equality holds if and only if $b=q$.
\end{proposition}

\begin{proof}
Put $s=x_2+\cdots+x_n$. For fixed $s$, there are
$\binom{s-1}{n-2}$ positive choices for $(x_2,\ldots,x_n)$.
There are at most $\frac bq(q-s)$ choices for $x_1$. Hence
\begin{align*}
 P_n(b,q,\ldots,q)
 &\le\frac bq\sum_{s=n-1}^{q-1}
 (q-s)\binom{s-1}{n-2}\\
 &=\frac bq\binom{q}{n}.
\end{align*}
The last identity follows by writing $q-s=\sum_{k=s}^{q-1}1$,
interchanging the sums, and applying the binomial summation formula twice.
Thus
\[
 n!P_n\le b\fall{q-1}{n-1}.
\]
On the other hand,
\begin{align*}
 g_n(b,q,\ldots,q)
 &=(b-q)(q-1)^{n-1}+\fall{q}{n},\\
 g_n-b\fall{q-1}{n-1}
 &=(b-q)\left((q-1)^{n-1}-\fall{q-1}{n-1}\right)\ge0.
\end{align*}
The expression in brackets is positive for $n\ge3$. Equality therefore
holds exactly when $b=q$.
\end{proof}

\section{Boundary terms in slice estimates}\label{sec:audit}

The next examples show that $g_d$ and the final falling-factorial term
can be negative on empty slices. They explain the role of the last
nonempty-slice index~\eqref{eq:M}.

\begin{example}\label{ex:41}
Take $n=4$ and $a_1=\cdots=a_4=41/10$. The slice $x_4=3$ has intercepts
\[
 \left(\frac{11}{10},\frac{11}{10},\frac{11}{10}\right).
\]
It is empty, since $3/(11/10)=30/11>1$. However,
\[
 g_3\left(\frac{11}{10},\frac{11}{10},\frac{11}{10}\right)
 =\fall{11/10}{3}=-\frac{99}{1000}.
\]
Here $P_3=0$, so $3!P_3>g_3$ on this slice. The nonempty-slice condition
excludes this case from the induction.
\end{example}

We next examine the finite difference between two successive slice
bounds. Fix $n\ge3$ and the ratios $r_1\ge\cdots\ge r_{n-1}\ge1$,
and let $u>0$ vary. Put
\[
 R(u)=\prod_{i=1}^{n-1}(r_i u-1).
\]
Using the polynomial extension of $g_n$, define
\begin{align*}
 D_n(u;\mathbf r)={}&
 g_n(r_1(u+1),\ldots,r_{n-1}(u+1),u+1)\\
 &-g_n(r_1u,\ldots,r_{n-1}u,u)
 -n g_{n-1}(r_1u,\ldots,r_{n-1}u).
\end{align*}
Direct expansion gives
\begin{align}
 D_n(u;\mathbf r)
={}&uR(u+1)-(u+n-1)R(u)-\bigl(u^n-(u-1)^n\bigr)\notag\\
&+n\fall{u}{n-1}+n(r_{n-1}u-1)^{n-1}
 -n\fall{r_{n-1}u}{n-1}.
\label{eq:interior-residual}
\end{align}
If $r_\ell u>1$ for every $1\le\ell\le n-1$, then, for
$1\le i\le n-2$,
\[
 \frac{\partial D_n}{\partial r_i}
 =u\left[(u+1)\prod_{\ell\ne i}(r_\ell(u+1)-1)
 -(u+n-1)\prod_{\ell\ne i}(r_\ell u-1)\right]>0.
\]
Each ratio of corresponding factors in the two products exceeds $1+1/u$, and the binomial expansion
gives $(u+1)(1+1/u)^{n-2}>u+n-1$.
For $n=4$, setting all three ratios equal to $s$ gives the exact
factorization
\[
 D_4(u;s,s,s)
 =u(s-1)\left(3s^2u+s^2+6su-2s+6u-4\right)\ge0
\]
for $u\ge1$ and $s\ge1$. Since $D_4$ increases with the first two
ratios, this gives $D_4(u;\mathbf r)\ge0$ for $u\ge1$ and $r_i\ge1$.
The boundary cases follow by continuity. This gives the finite-difference
bound. To apply the induction hypothesis, we also check that the slice
is nonempty.

\begin{example}\label{ex:51}
Take $n=5$ and $a_1=\cdots=a_5=51/10$. The slice $x_5=4$ has four equal
intercepts $11/10$ and is empty.
The final falling factorial of order five is
\[
 \fall{11/10}{5}=-\frac{54549}{100000}<0.
\]
This empty slice gives a negative final term. In the main proof, the
sum ends at the last nonempty slice, with index given by~\eqref{eq:M}.
\end{example}

\paragraph*{Author contributions.}
Z.L., S.S.-T.Y., and H.Z. designed research, performed research, and wrote the paper.
\paragraph*{Competing interests.}
The authors declare no competing interest.
\paragraph*{Data availability.}
All mathematical arguments are provided in this article and its appendices.
No external datasets are used.

\paragraph*{Acknowledgments.}
Yau is supported by the Tsinghua University Education Foundation.

The authors used ChatGPT (OpenAI; GPT-5.6 Sol and GPT-6 Astra in Pro mode), including through
the AI Mathematician (AIM) framework~\cite{AIM}, to assist with language editing,
mathematical exposition, and technical details of the proofs.
The authors reviewed and verified the AI-assisted material and take
full responsibility for the mathematical content of the manuscript.


\begin{thebibliography}{99}
\raggedright

\bibitem{YauZhang2006}
S.~T. Yau and L.~Zhang,
An upper estimate of integral points in real simplices with an
application to singularity theory.
\emph{Math. Res. Lett.} \textbf{13}, 911--921 (2006).
\href{https://doi.org/10.4310/MRL.2006.v13.n6.a6}{doi:10.4310/MRL.2006.v13.n6.a6}.

\bibitem{YauYuanZuo2016}
S.~S.-T. Yau, B.~Yuan, and H.~Zuo,
On the polynomial sharp upper estimate conjecture in
7-dimensional simplex.
\emph{J. Number Theory} \textbf{160}, 254--286 (2016).
\href{https://doi.org/10.1016/j.jnt.2015.08.012}{doi:10.1016/j.jnt.2015.08.012}.

\bibitem{XuYau1992}
Y.-J. Xu and S.~S.-T. Yau,
A sharp estimate of the number of integral points in a tetrahedron.
\emph{J. Reine Angew. Math.} \textbf{423}, 199--219 (1992).
\href{https://doi.org/10.1515/crll.1992.423.199}{doi:10.1515/crll.1992.423.199}.

\bibitem{XuYau1996}
Y.-J. Xu and S.~S.-T. Yau,
A sharp estimate of the number of integral points in a
4-dimensional tetrahedra.
\emph{J. Reine Angew. Math.} \textbf{473}, 1--23 (1996).
\href{https://doi.org/10.1515/crll.1995.473.1}{doi:10.1515/crll.1995.473.1}.

\bibitem{LuoYauZuo2015}
X.~Luo, S.~S.-T. Yau, and H.~Zuo,
A sharp polynomial estimate of positive integral points in a
4-dimensional tetrahedron and a sharp estimate of the Dickman--de Bruijn
function.
\emph{Math. Nachr.} \textbf{288}, 61--75 (2015).
\href{https://doi.org/10.1002/mana.201010081}{doi:10.1002/mana.201010081}.

\bibitem{LinLuoYauZuo2014}
K.-P. Lin, X.~Luo, S.~S.-T. Yau, and H.~Zuo,
On a number-theoretic conjecture on positive integral points in a
5-dimensional tetrahedron and a sharp estimate of the Dickman--de Bruijn
function.
\emph{J. Eur. Math. Soc.} \textbf{16}, 1937--1966 (2014).
\href{https://doi.org/10.4171/JEMS/480}{doi:10.4171/JEMS/480}.

\bibitem{LiangYauZuo2016}
A.~Liang, S.~S.-T. Yau, and H.~Zuo,
A sharp estimate of positive integral points in 6-dimensional
polyhedra and a sharp estimate of smooth numbers.
\emph{Sci. China Math.} \textbf{59}, 425--444 (2016).
\href{https://doi.org/10.1007/s11425-015-5061-x}{doi:10.1007/s11425-015-5061-x}.

\bibitem{LinYangYauZuo2016}
K.-P. Lin, A.~Yang, S.~S.-T. Yau, and H.~Zuo,
On the sharp polynomial upper estimate conjecture in
eight-dimensional simplex.
\emph{Pure Appl. Math. Q.} \textbf{12}, 353--398 (2016).
\href{https://doi.org/10.4310/PAMQ.2016.v12.n3.a2}{doi:10.4310/PAMQ.2016.v12.n3.a2}.

\bibitem{WangYau2007}
X.~Wang and S.~Yau,
On the GLY Conjecture of upper estimate of positive integral points
in real right-angled simplices.
\emph{J. Number Theory} \textbf{122}, 184--210 (2007).
\href{https://doi.org/10.1016/j.jnt.2006.04.007}
{doi:10.1016/j.jnt.2006.04.007}.

\bibitem{FeffermanSeco1995}
C.~L. Fefferman and L.~A. Seco,
The mathematics of large atoms.
\emph{Journ\'ees \'Equations aux D\'eriv\'ees Partielles}, no.~11, 1--12 (1995).
\href{https://doi.org/10.5802/jedp.483}{doi:10.5802/jedp.483}.

\bibitem{CordobaFeffermanSeco1995}
A.~C\'ordoba, C.~L. Fefferman, and L.~A. Seco,
Weyl sums and atomic energy oscillations.
\emph{Rev. Mat. Iberoamericana} \textbf{11}, 165--226 (1995).
\href{https://doi.org/10.4171/RMI/170}{doi:10.4171/RMI/170}.

\bibitem{Larson2019}
S.~Larson,
Maximizing Riesz means of anisotropic harmonic oscillators.
\emph{Ark. Mat.} \textbf{57}, 129--155 (2019).
\href{https://doi.org/10.4310/ARKIV.2019.v57.n1.a8}{doi:10.4310/ARKIV.2019.v57.n1.a8}.

\bibitem{Ennola1969}
V.~Ennola,
On numbers with small prime divisors.
\emph{Ann. Acad. Sci. Fenn. Ser. A I Math.} \textbf{440}, 1--16 (1969).
\href{https://doi.org/10.5186/aasfm.1969.440}{doi:10.5186/aasfm.1969.440}.

\bibitem{Borda2018}
B.~Borda,
Lattice points in algebraic cross-polytopes and simplices.
\emph{Discrete Comput. Geom.} \textbf{60}, 145--169 (2018).
\href{https://doi.org/10.1007/s00454-017-9946-z}{doi:10.1007/s00454-017-9946-z}.

\bibitem{AIM}
Yuanhang Liu, Yanxing Huang, Yanqiao Wang, Peng Li, and Yang Liu,
\emph{AI Mathematician: Towards Fully Automated Frontier Mathematical
Research},
arXiv:2505.22451, 2025.

\end{thebibliography}
\end{document}